\documentclass[reqno, 12pt]{amsart}

\pdfoutput=1
\makeatletter
\let\origsection=\section \def\section{\@ifstar{\origsection*}{\mysection}} 
\def\mysection{\@startsection{section}{1}\z@{.7\linespacing\@plus\linespacing}{.5\linespacing}{\normalfont\scshape\centering\S}}
\makeatother        

\usepackage{amsmath,amssymb,amsthm}
\usepackage{mathrsfs}
\usepackage{mathtools}
\usepackage{mathabx}\changenotsign
\usepackage{dsfont}
\usepackage{nicefrac}
\usepackage[normalem]{ulem}

\usepackage{graphicx}
\usepackage{xcolor}
\input{_preamble/_colours}

\usepackage[foot]{amsaddr}

\usepackage{verbatim}

\usepackage{enumitem}

\usepackage{hyperref}
\usepackage{cleveref}
\hypersetup{
    colorlinks,
    linkcolor={red!60!black},
    citecolor={green!60!black},
    urlcolor={blue!60!black},
}

\usepackage[open,openlevel=2,atend]{bookmark}

\input{_preamble/myBiblatex}

\defbibheading{bibliography}{%
  \section*{References}%
}

\usepackage[T1]{fontenc}
\usepackage{lmodern}
\usepackage[babel]{microtype}
\usepackage[english]{babel}

\usepackage{geometry}
\numberwithin{equation}{section}
\numberwithin{figure}{section}

\let\polishlcross=\l
\def\l{\ifmmode\ell\else\polishlcross\fi}

\def\paragraph#1{%
  \smallskip
  \noindent\textbf{#1.}\enspace}

\makeatletter
\def\moverlay{\mathpalette\mov@rlay}
\def\mov@rlay#1#2{\leavevmode\vtop{   \baselineskip\z@skip \lineskiplimit-\maxdimen
   \ialign{\hfil$\m@th#1##$\hfil\cr#2\crcr}}}
\newcommand{\charfusion}[3][\mathord]{
    #1{\ifx#1\mathop\vphantom{#2}\fi
        \mathpalette\mov@rlay{#2\cr#3}
      }
    \ifx#1\mathop\expandafter\displaylimits\fi}
\makeatother

\usepackage{thmtools, thm-restate}

\theoremstyle{plain}
\newtheorem{thm}{Theorem}[section]
    \crefname{thm}{Theorem}{Theorems}
\newtheorem{theorem}[thm]{Theorem}
    \crefname{theorem}{Theorem}{Theorems}
\newtheorem{lemma}[thm]{Lemma}
    \crefname{lemma}{Lemma}{Lemmas}
\newtheorem{corollary}[thm]{Corollary}
    \crefname{corollary}{Corollary}{Corollaries}

    \crefname{proposition}{Proposition}{Propositions}

    \crefname{problem}{Problem}{Problems}

    \crefname{conjecture}{Conjecture}{Conjectures}
\newtheorem{observation}[thm]{Observation}
    \crefname{observation}{Observation}{Observations}
\newtheorem{question}[thm]{Question}
    \crefname{question}{Question}{Questions}
\newtheorem*{claim*}{Claim}
\newtheorem{claim}{Claim}[]
    \crefname{claim}{Claim}{Claims}
\newtheorem*{case*}{Case}

    \crefname{case}{Case}{Case}
\newtheorem{thm-intro}{Theorem}[]
    \crefname{thm-intro}{Theorem}{Theorems}
\newtheorem{conj-intro}[thm-intro]{Conjecture}
    \crefname{conj-intro}{Conjecture}{Conjectures}
\newtheorem{question-intro}[thm-intro]{Question}
    \crefname{question-intro}{Question}{Questions}

\theoremstyle{definition}
\newtheorem{definition}[thm]{Definition}
    \crefname{definition}{Definition}{Definitions}

    \crefname{remark}{Remark}{Remarks}

    \crefname{remarks}{Remarks}{Remarks}

    \crefname{situation}{Situation}{Situations}

    \crefname{construction}{Construction}{Constructions}

    \crefname{construction}{Example}{Examples}
\newtheorem*{example*}{Example}

\newenvironment{proofofclaim}[1][Proof.]{%
    \begin{proof}[{#1}]%
        }{%
    \end{proof}}

\DeclareFontFamily{U}  {MnSymbolC}{}
\DeclareSymbolFont{MnSyC}         {U}  {MnSymbolC}{m}{n}
\DeclareFontShape{U}{MnSymbolC}{m}{n}{
    <-6>  MnSymbolC5
   <6-7>  MnSymbolC6
   <7-8>  MnSymbolC7
   <8-9>  MnSymbolC8
   <9-10> MnSymbolC9
  <10-12> MnSymbolC10
  <12->   MnSymbolC12}{}
\DeclareMathSymbol{\powerset}{\mathord}{MnSyC}{180}

\let\emptyset=\varnothing
\let\setminus=\smallsetminus

\newcommand*{\abs}[1]{\ensuremath{{\lvert {#1} \rvert}}}
\newcommand*{\gen}[1]{\ensuremath{{\langle {#1} \rangle}}}

\input{_preamble/myTikz.tex}

\hypersetup{
    pdftitle={Wild generalised truncations of infinite matroids},
    pdfauthor={J. Pascal Gollin, Attila Jo\'{o}}
}

\theoremstyle{plain}
\newtheorem*{martins}{Martin's Axiom}

\title{Wild generalised truncation of infinite matroids}

\author{J.~Pascal Gollin}
\thanks{J.~Pascal Gollin was supported by the Slovenian Research and Innovation Agency (research project N1-0370). Attila Jo\'{o} was supported by the Deutsche Forschungsgemeinschaft (DFG, German Research Foundation) - Grant No. 513023562.}
\address{J.~Pascal Gollin,
FAMNIT, University of Primorska, Koper, Slovenia}
\email{pascal.gollin@famnit.upr.si}
\author{Attila Jo\'{o}}
\address{Attila Jo\'{o},
Department of Mathematics, University of Hamburg, Hamburg, Germany}
\email{attila.joo@uni-hamburg.de}


\begin{document}

\begin{abstract}
    For ${n \in \mathbb{N}}$, the $n$-truncation of a matroid~$M$ of rank at least~$n$ is the matroid whose bases are the $n$-element 
    independent sets of~$M$. 
    One can extend this definition to negative integers by letting the $(-n)$-truncation be the matroid whose bases are all the sets that can be obtained by deleting~$n$ elements of a base of~$M$. 
    If~$M$ has infinite rank, then for distinct~${m,n \in \mathbb{Z}}$ the $m$-truncation and the $n$-truncation are distinct matroids. 
    
    Inspired by the work of Bowler and Geschke on infinite uniform matroids, we provide a natural definition of generalised truncations that 
    encompasses the notions mentioned above. 
    We call a generalised truncation wild if it is not an $n$-truncation for any~${n \in \mathbb{Z}}$ 
    and we prove that, under Martin's Axiom, any finitary matroid of infinite rank and size of less than continuum admits~${2^{2^{\aleph_0}}}$ wild generalised truncations. 
\end{abstract}

\maketitle

\section{Introduction}

Searching for a concept of infinite matroids with duality was initiated by Rado~\cite{rado1966abstract}. 
Rado's questions inspired Higgs, Oxley and others to 
investigate possible definitions~\cite{higgs1969matroids, oxley1978infinite}. 
The theory of infinite matroids gained a new momentum when Bruhn et al.~\cite{bruhn2013axioms} rediscovered independently and axiomatised the same infinite matroid concept that was found by Higgs in the late 1960s. 
Going beyond the work of Higgs, they gave five sets of cryptomorphic axiomatisation. 
Their axiomatisation in terms of bases reads as follows. 

A set~${\mathcal{B} \subseteq \mathcal{P}(E)}$ is the \emph{set of basis of a matroid} $M$ on a given ground set~$E$ if
\begin{enumerate}
    [label=(B\arabic*)]
    \item \label{item:B1} $\mathcal{B} \neq \emptyset$;
    \item \label{item:B2} For all~${B_0, B_1 \in \mathcal{B}}$ and all~${x \in B_0 \setminus B_1}$ there exists an element~${y \in B_1 
    \setminus B_0}$ 
    such that ${(B_0 \setminus \{x\}) \cup \{y\} \in \mathcal{B}}$;
\end{enumerate}
\begin{enumerate}
    [label=(BM)]
    \item \label{item:BM} For every~${X \subseteq E}$, the set of maximal elements of the poset ${(\{ X \cap B \colon\ B \in \mathcal{B} \}, 
    \subseteq)}$ form 
    a cofinal subset.
\end{enumerate}

The \emph{truncation} of a matroid~$M$ of non-zero rank is the 
matroid on the same ground set~$E$ 
whose bases are all sets that can be obtained by the deletion of one element of a base of~$M$ (see \cite[Definition 3.1]{bowler2014matroids}). 
If the rank of~$M$ is at least~${n \in \mathbb{N} \setminus \{0\}}$, then truncation can be iterated~$n$ times starting with~$M$. 
Let us call the resulting matroid the \emph{${(-n)}$-truncation} of~$M$. 
Note that if the rank of~$M$ is infinite, then so is the rank of its $(-n)$-truncation. 
For~${n \in \mathbb{N}}$, the \emph{$n$-truncation} of a matroid~$M$ of rank at least~$n$ is the matroid on the same ground set whose bases are the $n$-element independent sets of~$M$. 
Clearly, if~$M$ has infinite rank and~${m,n \in \mathbb{Z} \setminus \{0\}}$ with~${m \neq n}$, then the $m$-truncation and $n$-truncation of~$M$ are different matroids\footnote{Observe that from the viewpoint of the $(-n)$-truncation, the ``$(-0)$-truncation'' could also be defined as the matroid~$M$ itself. To avoid an overload in notation, we do not call~$M$ the $(-0)$-truncation but refer to it as the \emph{trivial} truncation. }.

Our aim of this paper is to find a natural common generalisation of these truncation operations that allows for more flexibility `in between' these concepts in a similar sense as Bowler and Geschke~\cite{bowler2016self} generalised the concept of uniform matroids, as we will discuss further below. 
Let~$E(M)$, $\mathcal{I}(M)$, and $\mathcal{B}(M)$ stand for the ground 
set, independent sets, and the bases of matroid~$M$ respectively. 
We propose the following definition. 

\begin{definition}
    A matroid~$N$ is a \emph{generalised truncation} of matroid~$M$ if 
    \begin{enumerate}[label=(\Roman*)]
        \item \label{item:GT1} ${E(N) = E(M)}$,
        \item \label{item:GT2} ${\mathcal{I}(N) \subseteq \mathcal{I}(M)}$,
        \item \label{item:GT3} for all~${I \in \mathcal{I}(N) \setminus \mathcal{B}(N)}$ and all~${e \in E \setminus I}$, 
            if~${I \cup \{ e \} \in \mathcal{I}(M)}$, 
            then~${I \cup \{ e \} \in \mathcal{I}(N)}$.
    \end{enumerate}
\end{definition}

\noindent
Note that every matroid is a generalised truncation of itself which we call the \emph{trivial} generalised truncation.
One can ask if there are non-trivial generalised truncations other than the $n$-truncation for~${n \in \mathbb{Z}}$. 
The generalised truncations of free matroids are exactly the uniform matroids (in the sense of the definition of Bowler and Geschke~\cite[Definition~2]{bowler2016self}). 
A uniform matroid~$U$ is \emph{wild}\footnote{In general, a 
matroid is called \emph{wild} if it is not \emph{tame}, i.e. it admits a circuit and a cocircuit with infinite intersection.} if neither~$U$ nor its dual has finite rank. 
Using our terminology, a uniform matroid is wild if it is neither a free matroid nor an $n$-truncation of a free matroid for suitable~${n \in \mathbb{Z}}$. 
Wild uniform matroids were constructed by Bowler and Geschke~\cite{bowler2016self} under Martin's Axiom. 
It is unknown if their existence can be proved in {ZFC} alone.  
We will call a non-trivial generalised truncation \emph{wild} if it is not an 
$n$-truncation for any~${n \in \mathbb{Z}}$. 
The main result of this paper reads as follows.

\begin{theorem}
    \label{thm: main Martin}
    Under Martin's Axiom, every finitary matroid~${M}$ of infinite rank on a ground set~$E$ with ${\abs{E} < 2^{\aleph_0}}$ admits a wild generalised truncation. 
\end{theorem}

\section{Preliminaries}

\subsection{Infinite matroids}


For~${X \subseteq E(M)}$, let~${M{\upharpoonright}X}$ be the matroid on~$X$ where ${\mathcal{B}(M {\upharpoonright} X)}$ consists of the maximal elements of~${\{ X \cap B \colon\ B \in \mathcal{B}(M) \}}$. 
It is known that ${M{\upharpoonright}X}$ is indeed a matroid, and it is called the \emph{restriction} of~$M$ to~$X$. Similarly, ${M.X}$, is the matroid on~$X$ 
where~${\mathcal{B}(M.X)}$ consists of the minimal elements of~${\{ X \cap B \colon\ B \in \mathcal{B}(M) \}}$, and it is called the \emph{contraction} of ~$M$ to~$X$. 
We write~${M \setminus X}$ and~${M/X}$ for~${M{\upharpoonright}(E \setminus X)}$ and~${M.(E \setminus X)}$ respectively. 
Their respective names are the \emph{deletion} and the \emph{contraction} of~$X$ in~$M$. 

If a matroid has a finite base, then all of its bases are of the same size. 
More generally, it follows from \ref{item:B2} that:
\begin{lemma}
    \label{lem: quasi equivcar}
    If~$M$ is a matroid and~${B,B' \in \mathcal{B}(M)}$ with ${\abs{B \setminus B'} < \aleph_0}$, then ${\abs{B' \setminus B} = \abs{B \setminus B'}}$. 
\end{lemma}
 
It is independent of ZFC if the bases of a fixed matroid have the same cardinality (see~\cite{higgs1969equicardinality} combined with 
\cite[Theorem~15]{bowler2016self}). 
Therefore, in the definition of the rank~$r(M)$ of matroid~$M$ the usage of cardinalities is avoided. 
The \emph{rank}~$r(M)$ of matroid~$M$ is~${n \in \mathbb{N}}$ if it has a base of size~$n$ and~$\infty$ if its bases are infinite. 
For~${X, Y \subseteq E(M)}$, we write~${r_M(X)}$ for~${r(M{\upharpoonright}X)}$ and call it the \emph{rank} of~$X$.
Moreover, ${r_M(X | Y)}$ stands for~${r_{M/Y}(X \setminus Y)}$ and it is 
called the \emph{relative rank} of~$X$ with respect to~$Y$. 
We will make use of the following observation for relative ranks (see~\cite{bruhn2013axioms}).
\begin{enumerate}
    [label=(R3)]
    \item \label{axiom:r3} For all~$C \subseteq B \subseteq A \subseteq E(M)$, we have~$r_M(A|C) = r_M(B|C) + r_M(A|B)$. 
\end{enumerate}
A set~${X \subseteq E}$ \emph{spans} ${e \in E}$ if~${r_M(\{ e \}| X ) = 0}$. 
We write~$\mathsf{span}_M(X)$ for the set of elements spanned by~$X$. 
A matroid~$M$ is \emph{finitary} if whenever all finite subsets of a set~${X \subseteq E}$ are independent in~$M$, then so is~$X$. 
For more about infinite matroids we refer to \cite{nathanhabil}.

\subsection{Martin's Axiom}

Let~${(P, \leq)}$ be a partial order. 
A set~${D \subseteq P}$ is \emph{dense} if every~${p \in P}$ has a lower bound in~$D$, that is there exists a~${d \in D}$ with~${d \leq p}$. 
A set~${A \subseteq P}$ is a \emph{strong antichain} if no two distinct elements of~$A$ have a common lower bound in~$P$. 
We say that~$P$ satisfies the \emph{countable chain condition} (or, $\mathsf{ccc}$, for short) if every strong antichain in~$P$ is countable. 
A non-empty set~${F \subseteq P}$ is a \emph{filter} if
\begin{itemize}
    \item $F$ is downward directed, that is any two distinct elements of~$F$ have a common lower bound in~$F$, and
    \item $F$ is upwards closed, that is for every~${f \in F}$ and every~${p \in P}$, if~${f \leq p}$, then~${p \in F}$. 
\end{itemize}
Let~$\mathfrak{c}$ denote~$2^{\aleph_0}$, the size of the continuum. 
\begin{martins}
    For every partial order~$(P,\leq)$ that satisfies the $\mathsf{ccc}$, every set~$I$ with~${\abs{I} < \mathfrak{c}}$, and every family~${\gen{ 
    D_i \colon i \in I}}$ of dense subsets of~$P$ there exists a filter~$F$ on~$P$ such that~${F \cap D_i}$ is non-empty for every~${i \in I}$. 
\end{martins}

\begin{theorem}[\cite{jech2002set}]
    \label{clm: card ar under Martin}
    Under Martin's Axiom, ${2^{\kappa} = \mathfrak{c}}$ for every cardinal~$\kappa$ with~${\aleph_0 \leq \kappa < \mathfrak{c}}$. 
\end{theorem}

\section{Preparations}

First, we give a characterisation for a set~$\mathcal{F}$ to be the set of bases of a generalised truncation of a matroid~$M$. 
\begin{lemma}
    \label{lem:wildcharacter}
    A set~$\mathcal{F}$ is the set of the bases of a generalised truncation of a matroid~$M$ if and only if it satisfies the following conditions:
    \begin{enumerate}
        [label=(\arabic*)]
        \item\label{item: indep} ${\emptyset \neq \mathcal{F} \subseteq \mathcal{I}(M)}$;
        \item\label{item: class} If~${B \in \mathcal{F}}$ and~${B' \in \mathcal{I}(M)}$ with~${\abs{B \setminus B'} = \abs{B' \setminus 
        B} < \aleph_0}$, then~${B' \in \mathcal{F}}$; 
        \item\label{item: exchange works} If~${B,B' \in \mathcal{F}}$, then no proper subset of~$B$ spans~$B'$ in~$M$;
        \item\label{item: mino axiom works} For every~${I, J \in \mathcal{I}(M)}$  with~${I \subseteq J}$, if there is a~${B \in \mathcal{F}}$ 
        with~${I \subseteq B}$, then there is a ${B' \in \mathcal{F}}$ such that either~${I \subseteq B' \subseteq J}$ or~${B' \supseteq J}$. 
    \end{enumerate}
\end{lemma}

\begin{proof}
    First, let us show that the set of bases~${\mathcal{F} \coloneqq \mathcal{B}(N)}$ of any generalised truncation~$N$ of a matroid~$M$ satisfies 
    these conditions. 
    Clearly, by \ref{item:B1} and \ref{item:GT2}, condition~\ref{item: indep} holds. 
    To show condition~\ref{item: class}, let~${B \in 
    \mathcal{F}}$ and~${B' \in \mathcal{I}(M)}$ with~${\abs{B \setminus B'} = \abs{B' \setminus B} < \aleph_0}$. 
    If~${B' \notin \mathcal{I}(N)}$, then by applying axiom \ref{item:BM} we get a maximal $N$-independent subset~$B''$ of~$B'$. 
    But then, \ref{item:GT3} ensures that~$B''$ is a base of~$N$, which contradicts \Cref{lem: quasi equivcar} since~${B'' \subsetneq B'}$. 
    Therefore, ${B' \in \mathcal{I}(N)}$. 
    If~${B' \notin \mathcal{F}}$, then a proper superset of~$B'$ is in~$\mathcal{F}$ which again contradicts \Cref{lem: quasi equivcar}. 
    Hence~${B' \in \mathcal{F}}$ and condition \ref{item: class} holds. 
    If there are~${I \subsetneq B \in \mathcal{F}}$ such that~$I$ spans a~${B' \in \mathcal{F}}$ in~$M$, then \ref{item:GT2} ensures that~$I$ spans~$B'$ in~$N$ as well. 
    Therefore~$I$ is a base of~$N$. 
    But then a proper subset of a base is a base which contradicts \Cref{lem: quasi equivcar}. 
    Therefore condition~\ref{item: exchange works} is satisfied. 
    For condition~\ref{item: mino axiom works}, let~${I,J \in \mathcal{I}(M)}$ with~${I \subseteq J}$ and suppose there is a~${B \in \mathcal{F}}$ with~${I \subseteq B}$. 
    Since~$N$ is a matroid, apply \ref{item:BM} with~${X \coloneqq J}$, so there is a maximal $N$-independent subset~${I' \subseteq J}$ that includes~$I$. 
    Suppose~${I' \notin \mathcal{F}}$. 
    Then by \ref{item:GT3}, there is no~${e \in J \setminus I'}$ as that would contradict the maximality of~$I'$, so~${I' = J}$ and hence~${J \in 
    \mathcal{I}(N)}$ and there is a base of~$N$ that includes~$J$, as required. Therefore \ref{item: mino axiom works} holds. 

    So let us show that these conditions are also sufficient. 
    We first prove that a set~$\mathcal{F}$ satisfying these condition is indeed the set of bases of some matroid. 
    Clearly, \ref{item:B1} is satisfied by condition \ref{item: indep}. 
    For \ref{item:B2}, let~${B,B' \in \mathcal{F}}$ and~${x \in B \setminus B'}$. 
    By \ref{item: exchange works} there is a~${y \in B'}$ such that~${y \notin \mathsf{span}_M(B \setminus \{x\})}$. 
    Hence, ${(B \setminus \{x\}) \cup \{y\} \in \mathcal{I}(M)}$ and, by condition~\ref{item: class}, ${(B \setminus \{x\}) \cup \{y\} \in 
    \mathcal{F}}$. 
    For \ref{item:BM}, let~${X \subseteq E}$, let~${B \in \mathcal{F}}$, and consider~${I \coloneqq X \cap B}$. 
    Since~$M$ is a matroid, let~${J}$  be a maximal  $M$-independent subset of~$X$ that includes~$I$. 
    By condition \ref{item: mino axiom works}, there is a~${B' \in \mathcal{F}}$ such that either~${I \subseteq B' \subseteq J}$ or~${B' \supseteq 
    J}$. 
    In both cases, ${X \cap B' \supseteq X \cap B}$, and~${X \cap B'}$ is a maximal element of~${\{ X \cap B'' \colon B'' \in 
    \mathcal{F} \}}$, proving~\ref{item:BM}. 
    Therefore,~$\mathcal{F}$ is the set of the bases of a matroid~$N$ on~$E$. 
    
    Hence, all that is left to show is that this~$N$ is indeed a generalised truncation of~$M$. 
    Clearly, \ref{item:GT1} holds and \ref{item:GT2} is true by \ref{item: indep}. 
    For \ref{item:GT3}, consider~${I \in \mathcal{I}(N) \setminus \mathcal{B}(N)}$ and~${e \in E \setminus I}$ such that~${I \cup \{e\} \in 
    \mathcal{I}(M)}$.
    By condition \ref{item: mino axiom works}, there is a ${B' \in \mathcal{F}}$ such that either~${I \subseteq B' \subseteq I \cup \{e\}}$ or~${B' \supseteq I\cup \{ e \}}$. 
    If the former case holds, then as~${I \notin \mathcal{F}}$, we have~${B' = I \cup \{e\}}$, so in both cases, $B'$ witnesses that~${I \cup \{e\} \in \mathcal{I}(N)}$, as required. 
\end{proof}

For the remainder of this section, we fix a matroid~${M}$ of infinite rank and set~${E \coloneqq E(M)}$ and~${\mathcal{I} \coloneqq \mathcal{I}(M)}$. 
For~${I, J \in \mathcal{I}}$, we say that~$J$ \emph{almost spans}~$I$, and write~${I  \trianglelefteq J}$, if~${r_{M}(I | J) < \infty}$. 
This defines a pre-order~$\trianglelefteq$ on~$\mathcal{I}$. 
For~${I, J \in \mathcal{I}}$, we say that~$I$ and~$J$ are \emph{weakly equivalent} if~${I \trianglelefteq J}$ and~${J \trianglelefteq I}$. 
Moreover, we say that~$I$ and~$J$ are \emph{strongly equivalent}, and write~${I \sim J}$, if~${r_{M}(I | J) = r_{M}(J | I) < \infty}$.
While clearly the relation of strong equivalence is reflexive and symmetric, to observe the transitivity we make use of the 
following lemma. 

\begin{lemma}
    \label{lem:strongcharacter}
    If~${I \sim J}$, then~${r_{M}(X|I) = r_{M}(X|J)}$ holds whenever~${I \cup 
    J \subseteq X \subseteq E}$. 
    Moreover, if there exists an~$X$ with~${I \cup J \subseteq X \subseteq E}$ such that~${r_{M}(X|I) = r_{M}(X|J) < \infty}$, then~${I \sim J}$. 
\end{lemma}

\begin{proof}
    Assume that~${I \sim J}$. 
    Using \ref{axiom:r3}, we get the following.   
    \begin{align*}
    r_{M}(X|I) 
        &= r_{M}(I\cup J| I)+r_{M}(X | I\cup J) 
        = r_{M}(J|I)+r_{M}(X | I\cup J)\\
        &= r_{M}(I|J)+r_{M}(X | I\cup J) 
        = r_{M}(I\cup J| I)+r_{M}(X | I\cup J)
        = r_{M}(X| J).
    \end{align*}

    Now suppose that~${r_{M}(X| I) = r_{M}(X| J) < \infty}$ where ${I \cup J \subseteq X \subseteq E}$. 
    Then, using \ref{axiom:r3}, 
    \[
        r_{M}(J|I) + r_{M}(X |I\cup J) 
        = r_{M}(X| I) 
        = r_{M}(X| J) 
        = r_{M}(I|J) + r_{M}(X | I \cup J).
    \]
    \noindent 
    Since~${r_{M}(X| I) < \infty}$, all these quantities are finite. 
    Hence by subtracting~${r_{M}(X | I \cup J)}$ from both sides we obtain~${r_{M}(I|J) = r_{M}(J|I) < \infty}$, 
    and therefore, ${I \sim J}$.
\end{proof}

\begin{corollary}
    Strong equivalence is an equivalence relation. 
\end{corollary}

\begin{proof}
    Reflexivity and symmetry are straightforward. 
    To prove the transitivity, assume that~${I \sim J}$ and~${J \sim K}$. Set~${X \coloneqq I \cup J \cup K}$. 
    Then~${r_{M}(X| I) < \infty}$ since (by using \ref{axiom:r3}) 
    \[
        r_{M}(X| I)
        = r_{M}(J | I) + r_{M}(X| I\cup J) 
        = r_{M}(J | I) + r_{M}(K| I\cup J) 
        \leq r_{M}(J| I) + r_{M}(K| J),
    \]
    where both summands on the right side are 
    finite by the assumptions~${I \sim J}$ and~${J \sim K}$. 
    Since~${I \sim J}$, \Cref{lem:strongcharacter} provides~${r_{M}(X| I) = r_{M}(X| J)}$. 
    Similarly, ${r_{M}(X| J) = r_{M}(X| K)}$. But then, ${r_{M}(X| I) = r_{M}(X| K) < \infty}$. 
    Again by \Cref{lem:strongcharacter}, we conclude that~${I \sim K}$.
\end{proof}

We observe the following characterisation of strong equivalence with finite difference. 

\begin{lemma}
    \label{lem: almost same}
    If~${I,J \in \mathcal{I}}$ with~${\abs{I \setminus J} < \aleph_0}$, then~${I \sim J}$ if and only if~${\abs{I \setminus J} = \abs{J \setminus I}}$. 
\end{lemma}

\begin{proof}
    Suppose first that~${\abs{I \setminus J} = \abs{J \setminus I} < \infty}$ and let~${X \coloneqq I \cup J}$. 
    Then \linebreak ${r_{M}(X|I) \leq \abs{J \setminus I} < \infty}$, and, furthermore, ${r_{M}(X|I \cap J) \leq \abs{J \setminus I} + \abs{I \setminus J} < \infty}$. 
    By \ref{axiom:r3} and since~${r_M(I|I \cap J) = \abs{I \setminus J}}$, we get ${r_{M}(X|I \cap J) = \abs{I \setminus J} + r_{M}(X|I)}$. 
    Similarly, ${r_{M}(X|I\cap J) = \abs{J \setminus I} + r_{M}(X|J)}$. 
    Hence
    \[ 
        r_{M}(X|I) 
        = r_{M}(X|I\cap J)-\abs{I \setminus J} 
        = r_{M}(X|I\cap J)-\abs{J \setminus I} 
        = r_{M}(X|J). 
    \]
    Thus~${I \sim J}$ follows by applying \Cref{lem:strongcharacter}.
    
    If on the other hand~${\abs{I \setminus J} < \aleph_0}$ and~${\abs{I \setminus J} \neq \abs{J \setminus I}}$, then there is either an~${I' \subsetneq I}$  with ${\abs{I' \setminus J} = \abs{J \setminus I'} < \aleph_0}$, or a ${J' \subsetneq  J}$ with ${\abs{I \setminus J'} = \abs{J' \setminus I} < \aleph_0}$. 
    Then we already know by the first part of the lemma that~${I' \sim J}$ and~${I \sim J'}$ in the respective cases. 
    Since no set is equivalent to a proper subset of itself, ${I \not\sim J}$ follows in both cases. 
\end{proof}

For ${I \in \mathcal{I}}$, let~${[I]}$ denote the equivalence class of~$I$ with respect to~$\sim$. 
Strong equivalence is clearly a refinement of the weak equivalence, thus it is compatible with almost spanning in the following sense. 

\begin{observation}
    \label{obs:compatible}
    Let~${I,I',J,J' \in \mathcal{I}}$. 
    If~${I \sim I'}$, ${J \sim J'}$, and~${I  \trianglelefteq J}$, then~${I'  \trianglelefteq J'}$. 
\end{observation}

Hence, the pre-order of almost spanning extends to a pre-order on the set of equivalence classes of~$\sim$. 
We abuse the notation by denoting this by~$ \trianglelefteq$ as well. 

Note that for a finite~${I \in \mathcal{I}}$ and every~${J \in \mathcal{I}}$, we have that~$J$ almost spans~$I$, and so~${[I] \trianglelefteq [J]}$. 
Similarly, if~${I \in \mathcal{I}}$ with~${r(M/I) < \infty}$, then~${[J]  \trianglelefteq [I]}$ for every~${J \in \mathcal{I}}$. 

\begin{lemma}
    \label{lem: finite and almost base}
    For every finite~${I \in \mathcal{I}}$, the set~${[I] = \{ J \in \mathcal{I} \colon\ \abs{J} = \abs{I} \}}$ consists of the bases of the $\abs{I}$-truncation of~$M$. 
    
    For every~${I \in \mathcal{I}}$ with~${r(M/I) \in \mathbb{N} \setminus \{0\}}$, the set~${[I] = \{ J \in \mathcal{I} \colon\ r(M/J) = r(M/I) \}}$ consists of the bases of the $(- r(M/I))$-truncation of~$M$. 
\end{lemma}

\begin{proof}
    Let~${n < \omega}$. 
    It follows directly from \Cref{lem: almost same} that~${\{ I \in \mathcal{I} \colon\ \abs{I} = n \}}$ is indeed an equivalence class of~$\sim$. 
    Furthermore, it is the~$\abs{I}$-truncation of~$M$ by definition. 
    
    By applying the second part of \Cref{lem:strongcharacter} with~${X = E}$, we conclude that the sets in ${\{ I \in \mathcal{I} \colon\ r(M/I) = n \}}$ are strongly equivalent. 
    Let~${J \in \mathcal{I}}$ with~${r(M/J) \neq n}$. 
    Then there is either an~${I' \supsetneq I}$ with~${r(M/I') = r(M/J) < \infty}$ or a~${J' \supsetneq J}$ with ${r(M/I) = r(M/J') < \infty}$ depending on if~${r(M/J) < r(M/I)}$ or~${r(M/J) > r(M/I)}$. 
    As before, ${I' \sim J}$ and~${I \sim J'}$ holds in the respective cases. 
    Since no set is strongly equivalent to a proper subset of itself, ${I \not\sim J}$ holds in both cases. 
\end{proof}

Finally, we will make use of the following technical lemma. 

\begin{lemma}
    \label{lem: ugly}
    If~${I,J \in \mathcal{I}}$ are infinite sets, then for every~${n < \omega}$ and finite~${J' \subseteq J}$, 
    there is a finite~${J'' \subseteq J \setminus J'}$ such that~${r_M(I|J \setminus J'') \geq n}$. 
\end{lemma}

\begin{proof}
    If~${I \cap J}$ is infinite, then any $n$-element subset of~${I \cap J}$ which is disjoint from~${J'}$ is suitable as~${J''}$. 
    
    Assume that~${I \cap J}$ is finite. 
    Then~${I \setminus J}$ is infinite. 
    Suppose first that~${J' = \emptyset}$. 
    Let ${e_0,\dots, e_{n-1} \in I \setminus J}$ be pairwise distinct. 
    It is easy to see, that there are~${f_0, \dots, f_{n-1} \in J \setminus I}$ such that~${(J \setminus \{ f_i \colon\ i < n \}) \cup \{ e_i \colon\ i < n \} \in \mathcal{I}}$. 
    But then, ${\{ e_i \colon\ i < n \}}$ witnesses that ${J'' \coloneqq \{f_i \colon\ i < n \}}$ is suitable. 
    Finally, if~${J' \neq \emptyset}$, then we replace~$J'$ by~$\emptyset$, $M$ by~${M/J'}$, $J$ by~${J \setminus J'}$, and~$I$ by a maximal ${M/J'}$-independent subset of~${I \setminus J'}$. 
    Clearly, the premise of the lemma still holds and the resulting~${J''}$ is suitable for the original setting as well. 
\end{proof}

\section{Main result}

We actually prove the following slight strengthening of \Cref{thm: main Martin}. 

\begin{theorem}
    \label{thm: main Martin extra}
    Let~$M$ be a finitary matroid of infinite rank on a ground set~$E$ with~${\abs{E} < \mathfrak{c}}$, 
    and let~$\mathcal{F}_0$ be the union of less than~$\mathfrak{c}$ many pairwise $\trianglelefteq$-incomparable equivalence classes of~$\sim$. 
    If Martin's Axiom holds, then there exists a generalised truncation~$N$ of~$M$ such that~${\mathcal{B}(N) \supseteq \mathcal{F}_0}$. 
\end{theorem}

\begin{proof}
    Without loss of generality, we may assume that~$\mathcal{F}_0$ is non-empty, as otherwise we just add the equivalence class of an 
    arbitrary~${I \in \mathcal{I} \coloneqq \mathcal{I}(M)}$. 
    If~${\mathcal{F}_0}$ contains a finite set~$I$ of size~$n$, then~${[I] \subseteq \mathcal{F}_0}$ as~$\mathcal{F}_0$ is closed under~$\sim$ by assumption. 
    Since~$[I]$ is a $\trianglelefteq$-smallest equivalence class and $\mathcal{F}_0$ is the union of pairwise $\trianglelefteq$-incomparable 
    equivalence classes, we must have~${\mathcal{F}_0 = [I]}$. 
    Since~$[I]$ is the set of bases of the $n$-truncation of~$M$ (see \Cref{lem: finite and almost base}), ${N \coloneqq (E,\mathcal{F}_0)}$ is suitable. 
    Similarly, suppose that~$\mathcal{F}_0$ contains an~$I$ with~${r(M/I) < \infty}$. 
    Then~${\mathcal{F}_0 = [I]}$ because~$[I]$ is a $\trianglelefteq$-largest equivalence class. 
    Since~$[I]$ is either~$\mathcal{B}(M)$ or the set of bases of the $(-n)$-truncation of~$M$ (see \Cref{lem: finite and almost base}), ${N \coloneqq (E,\mathcal{F}_0)}$ is suitable. 
    
    Suppose now that~$\mathcal{F}_0$ contains neither finite sets nor co-finite subsets of bases. 
    We define an increasing continuous sequence~${\left\langle \mathcal{F}_\alpha \colon\ \alpha < \mathfrak{c}  \right\rangle}$ with the intension that~${\mathcal{F} \coloneqq \bigcup_{\alpha < \mathfrak{c}} \mathcal{F}_{\alpha}}$ 
    is the set of bases of a generalised truncation~$N$ of~$M$. 
    Let~$\mathcal{T}$ consists of the ordered pairs~${(I,J)}$ of $M$-independent sets such that~$I$ is a co-infinite subset 
    of~$J$. 
    We think of a pair~${(I, J) \in \mathcal{T}}$ as the ``task'' that, if the bases of~$N$ we constructed so far contains a superset~$B$ of~$I$, to 
    add a~$B'$ for which either~${I \subseteq B' \subseteq J}$ or~${B' \supseteq J}$ in order to not violate~\ref{item: mino axiom works} of 
    \Cref{lem:wildcharacter}. 
    Let~${\left\langle (I_\alpha,J_{\alpha}) \colon\ \alpha < \mathfrak{c} \right\rangle}$ be a sequence with range~$\mathcal{T}$ in which every~${(I,J) \in \mathcal{T}}$ appears unbounded often. 
    
    For~${\alpha < \mathfrak{c}}$, we maintain: 
    \begin{enumerate}[label=(\roman*)]
        \item\label{item: Falpha indep} ${\mathcal{F}_\alpha \subseteq \mathcal{I}}$;
        \item\label{item: alpha tasks} $\mathcal{F}_\alpha$ satisfies the restriction of 
        \Cref{lem:wildcharacter}\ref{item: mino axiom works} to the pairs~${\{ (I_\beta, J_{\beta}) \colon\ \beta < \alpha \}}$;
        \item\label{item: incomp} if~${B,B' \in \mathcal{F}_\alpha}$ with ${B \not \sim B'}$, then~$B$ and~$B'$ are $\trianglelefteq$-incomparable; and
        \item\label{item: closed under equiv} ${\mathcal{F}_{\alpha+1} \setminus \mathcal{F}_{\alpha}}$ is either empty or an equivalence class of~$\sim$.
    \end{enumerate}
    
    Clearly, $\mathcal{F}_0$ as given by the input of the theorem is suitable. 
    Moreover, these conditions cannot be ruined at limit steps. 
    Suppose that~$\mathcal{F}_\alpha$ is defined for some~${\alpha < \mathfrak{c}}$ and satisfies the conditions. 
    
    If~$\mathcal{F}_\alpha$ satisfies the restriction of 
    \Cref{lem:wildcharacter}\ref{item: mino axiom works} to the pairs~${\{ (I_\beta, J_{\beta}) \colon\ \beta<\alpha+1 \}}$, then we let~${\mathcal{F}_{\alpha+1} \coloneqq \mathcal{F}_\alpha}$. 
    Suppose it does not. 
    
    \begin{claim}
        \label{clm: not spanned by I alpha}
        There is no~${B \in \mathcal{F}_\alpha}$ with~${B  \trianglelefteq I_\alpha}$.
    \end{claim}
    
    \begin{proofofclaim}
        Suppose for a contradiction that there is a~${B_0 \in \mathcal{F}_\alpha}$ with~${B_0  \trianglelefteq I_\alpha}$. 
        We know that there is a~${B_1 \in \mathcal{F}_\alpha}$ that includes~$I_\alpha$, since otherwise~$\mathcal{F}_\alpha$ satisfies the task~${(I_\alpha, J_\alpha)}$.  
        Clearly, ${B_0 \trianglelefteq B_1}$. 
        Then property \ref{item: incomp} ensures that~${B_0 \sim B_1}$ and hence~${B_1 \trianglelefteq I_\alpha}$. 
        It follows that ${B_1 \setminus  I_\alpha}$ must be finite. 
        Take a set~$B$ with~${I_\alpha \subseteq B \subseteq J_\alpha}$ such that~${\abs{B \setminus I_\alpha} = \abs{B_1 \setminus I_\alpha}}$. 
        Then~${B \sim B_1}$ by \Cref{lem: almost same} and hence~${B \in \mathcal{F}_\alpha}$ as~$\mathcal{F}_{\alpha}$ is closed under~$\sim$ by~\ref{item: closed under equiv}. 
        Since ${I_\alpha \subseteq B \subseteq J_\alpha}$, we conclude that~$\mathcal{F}_\alpha$ satisfies the task~${(I_\alpha, J_\alpha)}$, a contradiction. 
    \end{proofofclaim}
    
    \begin{claim}
        \label{clm: no J alpha span}
        There is no~${B \in \mathcal{F}_\alpha}$ with~${J_\alpha  \trianglelefteq B}$.
    \end{claim}
    
    \begin{proofofclaim}
        Suppose for a contradiction that there is a~${B \in \mathcal{F}_\alpha}$ with~${J_\alpha \trianglelefteq B}$. 
        Then there is an $M$-independent~${B' \supseteq B}$ 
        with~${n \coloneqq \abs{B' \setminus B} < \aleph_0}$ and~${J_\alpha  \subseteq \mathsf{span}_M(B') \eqqcolon X}$. 
        Let~$J_\alpha'$ be a base of~${M{\upharpoonright}X}$ that extends~$J_\alpha$. 
        Then any~${J \subseteq J_\alpha'}$ with~${\abs{J_\alpha' \setminus J} = n}$ is strongly equivalent to~$B$ (apply the second part of \Cref{lem:strongcharacter} with~$B$, $J$, and~$X$). 
        If~${\abs{J_{\alpha}' \setminus J_\alpha} \geq n}$, 
        then this leads to a~$J$ with~${J \sim B}$ and~${J \supseteq J_\alpha}$. 
        If~${\abs{J_{\alpha}' \setminus J_\alpha} < n }$, then this leads to a~$J$ with~${J \sim B}$ and~${I_\alpha \subseteq J \subsetneq J_\alpha}$. 
        Since~$\mathcal{F}_\alpha$ is closed under strong equivalence (see 
        \ref{item: closed under equiv}), it follows that~${J \in \mathcal{F}_\alpha}$. 
        In both cases, $\mathcal{F}_\alpha$ satisfies the task~${(I_\alpha, J_\alpha)}$, a contradiction.
    \end{proofofclaim}

    Let~${P \coloneqq \mathsf{Fn}(J_\alpha \setminus I_\alpha, 2)}$, i.e.~the poset of functions whose domain is a finite subset of~${J_\alpha \setminus I_\alpha}$ and whose range is a subset of~${\{ 0,1 \}}$, ordered by~$\supseteq$. 
    We pick a transversal~$\mathcal{R}$ of the equivalence classes included in~$\mathcal{F}_\alpha$. 
    Since~${\abs{\mathcal{F}_0} < \mathfrak{c}}$ by assumption, \ref{item: closed under equiv} ensures that~${\abs{\mathcal{R}} < \mathfrak{c}}$. 
    Let 
    \[
        \mathcal{R}_{I_\alpha} 
            \coloneqq \{ B \in \mathcal{R} \colon\  I_\alpha  \trianglelefteq B\}
        \quad\textnormal{ and }\quad
        \mathcal{R}^{J_\alpha} 
            \coloneqq \{ B \in \mathcal{R} \colon\ B \trianglelefteq J_\alpha \}.
    \]
    For~${B \in \mathcal{R}_{I_\alpha}}$ and~${n < \omega}$, let~${C_{B,n} \coloneqq \{ p \in P \colon\ r_{M}(p^{-1}(1)|B)\geq n \}}$. 
    For~${B \in \mathcal{R}^{J_\alpha}}$ and~${n < \omega}$, let~${D_{B,n} \coloneqq \{ p \in P \colon\ r_{M}(B| J_\alpha \setminus p^{-1}(0)) \geq n 
    \}}$. 
        
    \begin{claim}
         Each element of~${\mathcal{D} \coloneqq \{C_{B,n}, D_{B',n} \colon\ n < \omega, B \in \mathcal{R}_{I_\alpha}, B' \in \mathcal{R}^{J_\alpha} \}}$ is dense in~$P$.
     \end{claim}
     
     \begin{proofofclaim}
        Let~${p \in P}$ and~${n < \omega}$ be given. 
        Take a~${B \in \mathcal{R}_{I_\alpha}}$. 
        Then~${J_\alpha \setminus I_\alpha \not  \trianglelefteq B}$, since otherwise~${J_\alpha \trianglelefteq B}$ which contradicts \Cref{clm: no J alpha span}. 
        It follows that there is an infinite~${J \subseteq J_\alpha \setminus I_\alpha}$ with~${J \cap B = \emptyset}$ for which~${J \cup B \in \mathcal{I}}$. 
        Let~${e_0, \dots, e_{n-1} \in J \setminus \mathsf{dom}(p)}$. 
        For the function ${q \coloneqq p \cup \{ \left\langle e_i, 1  \right\rangle \colon\ i < n \}}$, we have~${q \leq p}$ and~${q \in C_{B,n}}$. 
        
        Take a~${B \in \mathcal{R}^{J_\alpha}}$. 
        Let~${B'}$ be a maximal ${M/I_\alpha}$-independent subset of~$B$. 
        \Cref{clm: not spanned by I alpha} guarantees that~$B'$ is infinite. 
        Then, by applying \Cref{lem: ugly} to~${M/I_\alpha}$ with ${I = B'}$, ${J = J_\alpha \setminus I_\alpha}$, and~${J' = \mathsf{dom}(p)}$, we obtain a finite~${F \subseteq J_\alpha \setminus I_\alpha \subseteq J_\alpha}$ with~${\mathsf{dom}(p) \cap F = \emptyset}$ such that ${r_M(B | J_\alpha \setminus F ) \geq r_M(B' | (J_\alpha \setminus I_\alpha) \setminus F) \geq n}$. 
        Then~${q \coloneqq p \cup \{ \left\langle e, 0 \right\rangle \colon\ e\in F \} \in D_{B,n}}$. 
     \end{proofofclaim}
     
     Since~${\abs{\mathcal{D}} < \mathfrak{c}}$ as~${\abs{\mathcal{R}} < \mathfrak{c}}$ and~$P$ is known to be~$\mathsf{ccc}$ (see \cite[Lemma 14.35]{jech2002set}), Martin's 
     Axiom guarantees that there exists a 
     filter~$F$ in~$P$  that has non-empty intersection with every element of~$\mathcal{D}$. 
     We set~${B_\alpha \coloneqq I_\alpha \cup \bigcup_{p \in F} p^{-1}(1)}$ and~${\mathcal{F}_{\alpha+1} \coloneqq \mathcal{F}_\alpha \cup [B_\alpha]}$. 
     
     \begin{claim}
          $\mathcal{F}_{\alpha+1}$ maintains the conditions \ref{item: Falpha indep}-\ref{item: closed under equiv}. 
      \end{claim}
      
    \begin{proofofclaim}
        We only show that \ref{item: incomp} remains true, the rest are straightforward from the construction. 
        We may assume that~${\mathcal{F}_{\alpha+1}\neq \mathcal{F}_{\alpha}}$, since otherwise there is nothing to prove. 
        Then~${\mathcal{F}_{\alpha+1} \setminus \mathcal{F}_{\alpha} = [B_\alpha]}$ by construction. 
        Let~${B \in \mathcal{F}_\alpha}$ be given. 
        We may assume without loss of generality that~${B \in \mathcal{R}}$ 
        because the relation~$\trianglelefteq$ is compatible with~$\sim$ (\Cref{obs:compatible}). 
    
        First we show that~${B_\alpha \not \trianglelefteq B}$. 
        If~${B \in \mathcal{R} \setminus \mathcal{R}_{I_\alpha}}$, then, by definition, ${I_\alpha \not \trianglelefteq B}$ and hence ${B_\alpha \not \trianglelefteq B}$ because~${B_\alpha \supseteq I_\alpha}$. 
        If~${B \in \mathcal{R}_{I_\alpha}}$, 
        then let~${n<\omega}$ be arbitrary and let~${p \in F \cap C_{B,n}}$.  Then ${r_{M}(B_\alpha| B) \geq r_{M}(p^{-1}(1)|B ) \geq n}$. 
        Since~${n < \omega}$ was arbitrary, ${r_{M}(B_\alpha| B) = \infty}$ follows, which means~${B_\alpha \not \trianglelefteq B}$. 
        
        We turn to the proof of~${B \not \trianglelefteq B_\alpha}$. 
        If~${B \in \mathcal{R} \setminus \mathcal{R}^{J_\alpha}}$, then, by definition, ${B \not \trianglelefteq J_\alpha}$ and hence~${B \not \trianglelefteq B_\alpha}$ since~${B_\alpha \subseteq J_\alpha}$. 
        If~${B \in \mathcal{R}^{J_\alpha}}$, then let~${n < \omega}$ be arbitrary and let~${p \in F \cap D_{B,n}}$. 
        Then~${J_\alpha \setminus p^{-1}(0) \supseteq B_\alpha}$ and hence
        ${r_{M}(B |B_{\alpha}) \geq r_{M}(B| J_\alpha \setminus p^{-1}(0)) \geq n}$. 
        Since~${n < \omega}$ was arbitrary, ${r_{M}(B |B_\alpha) = \infty}$ follows, which means~${B \not \trianglelefteq B_\alpha}$. 
    \end{proofofclaim}
    
    \begin{claim}
        $\mathcal{F} \coloneqq \bigcup_{\alpha<\mathfrak{c}}\mathcal{F}_\alpha$ is the bases of a generalized truncation of~$M$. 
    \end{claim}
    
    \begin{proofofclaim}
        We check the condition given in \Cref{lem:wildcharacter}. 
        Clearly, ${\mathcal{F} \neq \emptyset}$ because we assumed that~${\mathcal{F}_0 \neq \emptyset}$. 
        We have~${\mathcal{F} \subseteq \mathcal{I}}$ by construction. 
        Therefore, \ref{item: indep} holds. 
        If~${B, B'\in \mathcal{I}}$ with~${\abs{B\setminus B'} = \abs{B'\setminus B} < \aleph_0}$, then~${B \sim B'}$ by \Cref{lem: almost same}. 
        The set~$\mathcal{F}$ is closed under strong equivalence because~$\mathcal{F}_0$ is and we maintained \ref{item: closed under equiv}. 
        Thus \Cref{lem: almost same} implies that in particular \ref{item: class} holds. 
        Preserving~\ref{item: incomp} implies that~\ref{item: exchange works} is satisfied. 
        To check \ref{item: mino axiom works}, let ${I, J \in \mathcal{I}}$ with~${I \subseteq J}$ and suppose that there is a~${B \in \mathcal{F}}$ with~${I \subseteq B}$.
        If~${\abs{J \setminus I} < \aleph_0}$, then \Cref{lem: almost same} provides a~$B'$ with~${B' \sim B}$ such that either~${I \subseteq B' \subseteq J}$ or~${B' \supsetneq J}$ depending on if~${\abs{B \setminus I} \leq \abs{J \setminus I}}$ or~${\abs{B \setminus I} > \abs{J \setminus I}}$. 
        By \ref{item: class}, we have~${B' \in \mathcal{F}}$. 
        Suppose that~${\abs{J \setminus I} \geq \aleph_0}$. 
        Let~${\alpha < \mathfrak{c}}$ be an ordinal such that~${B \in \mathcal{F}_\alpha}$. 
        Then, by construction, there is a~${\beta > \alpha}$ such that~${(I_\beta,J_\beta) = (I,J)}$. 
        But then \ref{item: alpha tasks} ensures that~$\mathcal{F}_{\beta+1}$ contains a~$B'$ such that either~${I \subseteq B' \subseteq J}$ or~${B' \supseteq J}$. 
    \end{proofofclaim}
    \noindent
    This concludes the proof of the theorem. 
\end{proof}

\Cref{thm: main Martin} then follows by taking an infinite and co-infinite subset~$B$ of a base of~$M$ and applying \Cref{thm: main Martin extra} with~${\mathcal{F}_0 \coloneqq [B]}$. 
Indeed, for the resulting matroid~$N$, ${B \in \mathcal{B}(N)}$ ensures that~$N$ is a wild generalised truncation of~$M$. 
 
\begin{corollary}
    If Martin's Axiom holds, then every finitary matroid~$M$ of infinite rank on a ground set~$E$ with~${\abs{E} < \mathfrak{c}}$ has exactly~$2^{\mathfrak{c}}$ pairwise non-isomorphic wild truncations. 
\end{corollary}

\begin{proof}
    First we show that there cannot be more. 
    To do so, it is enough to prove that there are at most~$2^{\mathfrak{c}}$ matroids on~$E$. 
    A matroid is uniquely determined by its bases and the set of bases is a subset of~$\mathcal{P}(E)$. 
    We have ${\abs{\mathcal{P}(E)} = 2^{\abs{E}} = \mathfrak{c}}$ by \Cref{clm: card ar under Martin}, 
    therefore~${\abs{\mathcal{P}(\mathcal{P}(E))} = 2^{\mathfrak{c}}}$ is indeed an upper bound for the number of all matroids on~$E$. 
    
    To find~$2^{\mathfrak{c}}$ wild truncations of~$M$, pick a~${B \in \mathcal{B}(M)}$. 
    For~${n < \omega}$ and~${i \in \{0,1\}}$, take an infinite set~${B_{n,i} \subseteq B}$ in such a way that~${B_{n,0} \subsetneq B_{n,1}}$ and~${B_{m,1} \cap B_{n,1} = \emptyset}$ for~${m \neq n}$. 
    Then, for every~${s \in 2^{\omega}}$, the set~${\mathcal{F}_s \coloneqq  \bigcup \{ [B_{n, s(n)}] \colon\ n < \omega\}}$ satisfies 
    the properties in the premise of \Cref{thm: main Martin extra} about~$\mathcal{F}_0$. 
    Moreover, for~${s \neq s'}$, ${\mathcal{F}_s \cup 
    \mathcal{F}_{s'}}$ cannot be extended to the set of bases of a matroid because it contains $\subseteq$-comparable elements. 
    It follows that applying \Cref{thm: main Martin extra} with~$\mathcal{F}_s$ and~$\mathcal{F}_{s'}$ results in different generalised truncations of~$M$. This shows that there are~$2^\mathfrak{c}$ 
    pairwise distinct generalised truncations. 
    Since there are~$\mathfrak{c}$ permutations of~$E$, among these matroids only~$\mathfrak{c}$ many can be isomorphic to each other. 
    Thus there must be~$2^{\mathfrak{c}}$ pairwise non-isomorphic generalised truncations of~$M$. 
\end{proof}

Our proof relies heavily on the assumption that matroid~$M$ is finitary. 
It seems natural to ask if it is possible to prove something for general matroids in a suitable setting. 

\begin{question}
    Is it consistent relative to ZFC that every matroid of infinite rank admits a wild generalised truncation?
\end{question}

\printbibliography

@article{bruhn2013axioms,
    AUTHOR = {Bruhn, Henning and Diestel, Reinhard and Kriesell, Matthias
              and Pendavingh, Rudi and Wollan, Paul},
     TITLE = {Axioms for infinite matroids},
   JOURNAL = {Adv. Math.},
  FJOURNAL = {Advances in Mathematics},
    VOLUME = {239},
      YEAR = {2013},
     PAGES = {18--46},
      ISSN = {0001-8708,1090-2082},
   MRCLASS = {05B35},
  MRNUMBER = {3045140},
MRREVIEWER = {Dillon\ Mayhew},
       DOI = {10.1016/j.aim.2013.01.011},
       URL = {https://doi.org/10.1016/j.aim.2013.01.011},
    EPRINT = {1003.3919},
EPRINTTYPE = {arxiv},
      NOTE = {},
}

@thesis{nathanhabil,
      TYPE = {Habilitation Thesis},
    AUTHOR = {Bowler, Nathan},
     TITLE = {Infinite matroids},
    SCHOOL = {University of Hamburg},
      YEAR = {2014},
       URL = {https://www.math.uni-hamburg.de/spag/dm/papers/Bowler_Habil.pdf},
      NOTE = {},
}

@article{bowler2014matroids,
    AUTHOR = {Bowler, Nathan and Carmesin, Johannes},
     TITLE = {Matroids with an infinite circuit-cocircuit intersection},
   JOURNAL = {J. Combin. Theory Ser. B},
  FJOURNAL = {Journal of Combinatorial Theory. Series B},
    VOLUME = {107},
      YEAR = {2014},
     PAGES = {78--91},
      ISSN = {0095-8956,1096-0902},
   MRCLASS = {05B35},
  MRNUMBER = {3213628},
MRREVIEWER = {Thomas\ Britz},
       DOI = {10.1016/j.jctb.2014.02.005},
       URL = {https://doi.org/10.1016/j.jctb.2014.02.005},
    EPRINT = {1202.3406},
EPRINTTYPE = {arxiv},
      NOTE = {},
}

@article{bowler2016self,
    AUTHOR = {Bowler, Nathan and Geschke, Stefan},
     TITLE = {Self-dual uniform matroids on infinite sets},
   JOURNAL = {Proc. Amer. Math. Soc.},
  FJOURNAL = {Proceedings of the American Mathematical Society},
    VOLUME = {144},
      YEAR = {2016},
    NUMBER = {2},
     PAGES = {459--471},
      ISSN = {0002-9939,1088-6826},
   MRCLASS = {05B35 (03E25 03E30 03E35 03E50)},
  MRNUMBER = {3430826},
MRREVIEWER = {Dillon\ Mayhew},
       DOI = {10.1090/proc/12667},
       URL = {https://www.math.uni-hamburg.de/home/geschke/papers/UniformMatroid7.pdf},
    EPRINT = {},
EPRINTTYPE = {},
      NOTE = {},
}

@article{higgs1969equicardinality,
    AUTHOR = {Higgs, Denis Arthur},
     TITLE = {Equicardinality of bases in {$B$}-matroids},
   JOURNAL = {Canad. Math. Bull.},
  FJOURNAL = {Canadian Mathematical Bulletin. Bulletin Canadien de
              Math\'ematiques},
    VOLUME = {12},
      YEAR = {1969},
     PAGES = {861--862},
      ISSN = {0008-4395,1496-4287},
   MRCLASS = {05.35},
  MRNUMBER = {253926},
MRREVIEWER = {R.\ A.\ Brualdi},
       DOI = {10.4153/CMB-1969-112-6},
       URL = {https://doi.org/10.4153/CMB-1969-112-6},
    EPRINT = {},
EPRINTTYPE = {},
      NOTE = {},
}

@article{higgs1969matroids,
    AUTHOR = {Higgs, Denis Arthur},
     TITLE = {Matroids and duality},
   JOURNAL = {Colloq. Math.},
  FJOURNAL = {Colloquium Mathematicum},
    VOLUME = {20},
      YEAR = {1969},
     PAGES = {215--220},
      ISSN = {0010-1354,1730-6302},
   MRCLASS = {05.35},
  MRNUMBER = {274315},
MRREVIEWER = {D.\ J. A. Welsh},
       DOI = {10.4064/cm-20-2-215-220},
       URL = {http://eudml.org/doc/267207},
    EPRINT = {http://eudml.org/doc/267207},
EPRINTTYPE = {url},
      NOTE = {},
}

@article{oxley1978infinite,
    AUTHOR = {Oxley, James G.},
     TITLE = {Infinite matroids},
   JOURNAL = {Proc. London Math. Soc. (3)},
  FJOURNAL = {Proceedings of the London Mathematical Society. Third Series},
    VOLUME = {37},
      YEAR = {1978},
    NUMBER = {2},
     PAGES = {259--272},
      ISSN = {0024-6115,1460-244X},
   MRCLASS = {05B35},
  MRNUMBER = {507607},
MRREVIEWER = {R.\ A.\ Brualdi},
       DOI = {10.1112/plms/s3-37.2.259},
       URL = {https://doi.org/10.1112/plms/s3-37.2.259},
    EPRINT = {},
EPRINTTYPE = {},
      NOTE = {},
}

@article{rado1966abstract,
    AUTHOR = {Rado, Richard},
     TITLE = {Abstract linear dependence},
   JOURNAL = {Colloq. Math.},
  FJOURNAL = {Colloquium Mathematicum},
    VOLUME = {14},
      YEAR = {1966},
     PAGES = {257--264},
      ISSN = {0010-1354,1730-6302},
   MRCLASS = {08.30},
  MRNUMBER = {184891},
MRREVIEWER = {G.\ Gr\"atzer},
       DOI = {10.4064/cm-14-1-257-264},
       URL = {https://doi.org/10.4064/cm-14-1-257-264},
}

@book{jech2002set,
    AUTHOR = {Jech, Thomas},
     TITLE = {Set theory},
   EDITION = {3rd millennium ed.},
 PUBLISHER = {Springer-Verlag, Berlin},
      YEAR = {2003},
     PAGES = {xiv+769},
      ISBN = {3-540-44085-2},
   MRCLASS = {03Exx (03-01 03-02)},
  MRNUMBER = {1940513},
MRREVIEWER = {Eva\ Coplakova},
       DOI = {10.1007/3-540-44761-X},
    EPRINT = {},
EPRINTTYPE = {},
      NOTE = {},
    SERIES = {Springer Monographs in Mathematics},
}

\end{document}